\documentclass{article}
\usepackage{amssymb,latexsym}
\usepackage{amsmath}
\usepackage{graphics}
\usepackage{graphicx}
\newtheorem{theorem}{Theorem}
\newtheorem{result}{Result}

\newtheorem{definition}{Definition}

\newenvironment{proof}[1][Proof]{\noindent\textbf{#1.} }{\ \rule{0.5em}{0.5em}}
\numberwithin{equation}{section}
\numberwithin{theorem}{section}
\numberwithin{corollary}{section}
\numberwithin{definition}{section}
\numberwithin{result}{section}

\begin{document}
\title{Timelike Bertrand Curves in Semi-Euclidean Space}
\author{Soley ERSOY and Murat TOSUN\\
sersoy@sakarya.edu.tr and tosun@sakarya.edu.tr\\
Department of Mathematics, Faculty of Arts and Sciences,\\
Sakarya University, Sakarya, 54187 TURKEY}

\maketitle
\begin{abstract}
In this paper, it is proved that, no special timelike Frenet curve
is a Bertrand curve in $\mathbb{E}_2^4$ and also, in
$\mathbb{E}_\nu^{n+1}$ $ \left( {n \ge 3} \right)$, such that the
notion of Bertrand curve is definite only in $\mathbb{E}_1^2$ and
$\mathbb{E}_1^3$. Therefore, a generalization of timelike Bertrand
curve is defined and called as timelike (1,3)-Bertrand curve in
$\mathbb{E}_2^4$. Moreover, the characterization of timelike
(1,3)-Bertrand curve is given in $\mathbb{E}_2^4$.


\end{abstract}

\section{Introduction}\label{S:intro}
In Euclidean ambient, a curve $C$ is called $ C^\infty -$special
Frenet curve in $\mathbb{E}^3$ if there exist three $ C^\infty
-$vector fields, that is, the unit vector tangent field ${\bf t}$,
the unit principal normal vector field  ${\bf n}$ and the unit
binormal vector field  ${\bf b}$ and two $ C^\infty -$scalar
functions, that is, the curvature function $ k_1 \left( { > 0}
\right)$ and torsion function $k_2 \left( { \ne 0} \right)$, \cite{Wo}.\\
In $\mathbb{E}^3$, a $ C^\infty -$special Frenet curve $C$ is
called a Bertrand curve if there exist another $ C^\infty
-$special Frenet curve $\bar C$ and a $ C^\infty -$mapping  $
\varphi :C \to \bar C$, such that the principal normal vector
fields of $C$ and $\bar C$ coincide at the corresponding points, \cite{Ma}.\\
If we focus our attention to timelike curves in three dimensional
Minkowski space $\mathbb{E}_1^3$, then we can give the well known
theorem that a $ C^\infty -$special timelike Frenet curve is a
Bertrand curve if and only if its curvature function $k_1$ and
torsion function $k_2$ satisfy $ak_1  + bk_2  = 1$ for all $s \in
L$, where $a$ and $b$ are constant real numbers. This theorem
suffices to define the timelike curve
\[
{\bf \bar c}\left( {\bar s} \right) = {\bf c}\left( {\varphi
\left( s \right)} \right) = {\bf c}\left( s \right) + \alpha {\bf
n}\left( s \right)
\]
then it is immediate that $\bar C$ is the timelike Bertrand mate
curve of $C$,\cite{Ba,Lo}. As a result of this theorem, in
$\mathbb{E}_1^3$ every timelike circular helix is an example of
timelike Bertrand curve. Moreover, any timelike planar curve is a
timelike Bertrand curve whose Bertrand mates are parallel curves
of $C$, \cite{Ba}.\\
In this paper, we prove that there is no $ C^\infty -$special
timelike Frenet curve which is a Bertrand curve in
$\mathbb{E}_2^4$ and also, in $\mathbb{E}_\nu^{n+1}$ $ \left( {n
\ge 3} \right)$. Thus, the notion of timelike Bertrand curve
stands only in $\mathbb{E}_1^2$ and $\mathbb{E}_1^3$. According to
this reason, we suggest an idea of generalization of timelike
Bertrand curve in $\mathbb{E}_2^4$.

\section{Preliminaries}\label{Sec2}
To meet the requirements in the next sections, the basic elements
of the theory of curves in the semi Euclidean space
$\mathbb{E}_2^4$ are briefly presented in this section. A more
complete elementary
treatment can be found in \cite{On}.\\
Semi-Euclidean space $\mathbb{E}_2^4$ is an Euclidean space
provided with standard flat metric given by
\[
g =  - dx_1^2  - dx_2^2  + dx_3^2  + dx_4^2
\]
where $\left( {x_1 ,x_2 ,x_3 ,x_4 } \right)$ is a rectangular
coordinate system in $\mathbb{E}_2^4$.\\
Since $g$ is an indefinite metric, recall that a vector ${\bf v}
\in \mathbb{E}_2^4$ can have one of three Lorentzian causal
characters: it can be spacelike if $g\left( {{\bf v},{\bf v}}
\right) > 0$ or ${\bf v}=0$, timelike if $g\left( {{\bf v},{\bf
v}} \right) < 0$ and null (lightlike) if $g\left( {{\bf v},{\bf
v}} \right) = 0$ and ${\bf v} \ne 0$, \cite{On}.\\
Similarly, an arbitrary curve ${\bf c} = {\bf c}\left( s \right)$
in $\mathbb{E}_2^4$ can locally be spacelike, timelike or null
(lightlike) if all of its velocity vectors $ {\bf c}'\left( s
\right)$ are, respectively, spacelike, timelike or null
(lightlike). The norm of ${\bf v} \in \mathbb{E}_2^4$ is defined
as $\left\| {\bf v} \right\| = \sqrt {\left| {g\left( {{\bf
v},{\bf v}} \right)} \right|}$. Therefore, ${\bf v}$ is a unit
vector if $g\left( {{\bf v},{\bf v}} \right) =  \mp 1$.
Furthermore, vectors ${\bf v}$ and ${\bf w}$ are said to be
orthogonal if $g\left( {{\bf v},{\bf w}} \right) = 0$. The
velocity of the curve $C$ is given by $ \left\| {{\bf c}'}
\right\|$. Thus, a timelike curve $C$ is said to be parametrized
by arc length function $s$ if  $g\left( {{\bf c}',{\bf c}'}
\right) = -1$, \cite{On}.\\
Let $ \left\{ {{\bf t}\left( s \right),{\bf n}_1 \left( s
\right),{\bf n}_2 \left( s \right),{\bf n}_3 \left( s \right)}
\right\}$ denotes the moving Frenet frame along $C$ in the
semi-Euclidean space $\mathbb{E}_2^4$, then ${\bf t},\;{\bf n}_1
,\;{\bf n}_2$, and ${\bf n}_3$ are called the tangent, the
principal normal, the first binormal, and the second
binormal vector fields of $C$, respectively.\\
Let $C$ be a  $ C^\infty -$special timelike Frenet curve with
timelike principal normal, spacelike first binormal and second
binormal vector fields in $\mathbb{E}_2^4$, parametrized by arc
length function $s$. Moreover, non-zero $ C^\infty -$scalar
functions $k_1$, $k_2$, and $k_3$ be the first, second, and third
curvatures of $C$, respectively. Then for the $ C^\infty -$special
timelike Frenet curve $C$, the following Frenet formula is given
by
\begin{equation}\label{2.1}
\left[ {\begin{array}{*{20}c}
   {{\bf t'}}  \\
   {{\bf n'}_{\bf 1} }  \\
   {{\bf n'}_{\bf 2} }  \\
   {{\bf n'}_{\bf 3} }  \\
\end{array}} \right] = \left[ {\begin{array}{*{20}c}
   0 & { - k_1 } & 0 & 0  \\
   {k_1 } & 0 & {k_2 } & 0  \\
   0 & {k_2 } & 0 & {k_3 }  \\
   0 & 0 & { - k_3 } & 0  \\
\end{array}} \right]\left[ {\begin{array}{*{20}c}
   {\bf t}  \\
   {{\bf n}_{\bf 1} }  \\
   {{\bf n}_{\bf 2} }  \\
   {{\bf n}_{\bf 3} }  \\
\end{array}} \right]
\end{equation}
where $ {\bf t},\;{\bf n}_1 ,\;{\bf n}_2 ,\;{\bf n}_3$ mutually
orthogonal vector fields satisfying
\[
{\rm g}\left( {{\bf t},{\bf t}} \right) = {\rm  g}\left( {{\bf
n}_1 ,{\bf n}_1 } \right) =  - 1\quad ,\;\quad {\rm g}\left( {{\bf
n}_2 ,{\bf n}_2 } \right) = {\rm g}\left( {{\bf n}_3 ,{\bf n}_3 }
\right) = 1
\]
(for the semi-Euclidean space $\mathbb{E}_\nu^{n+1}$, see
\cite{Co}).

\section{Timelike Bertrand Curve in $\mathbb{E}_2^4$}\label{Sec3}
The following two theorems related to Bertrand curves in
$\mathbb{E}_1^2$ and $\mathbb{E}_1^3$ are well known \cite{Ba,Lo}.

\begin{theorem}\label{T3.1.}
In $\mathbb{E}_1^2$, every timelike $ C^\infty -$planar curve is a
Bertrand curve,\cite{Ba}.
\end {theorem}

\begin{theorem}\label{T3.2.}
In $\mathbb{E}_1^3$, a $ C^\infty -$special timelike Frenet curve
with first and second curvatures $k_1$ and $k_2$ is a timelike
Bertrand curve if and only if there exist a linear relation  $ak_1
+ bk_2  = 1$ for all $s \in L$, where $ a,b$ are nonzero constant
real numbers, \cite{Ba}.
\end {theorem}
Now, let us investigate Bertrand curves in $\mathbb{E}_2^4$.

\begin{definition}\label{D3.1.}
A $C^\infty -$special timelike Frenet curve $C$ $\left( {C:L \to
\mathbb{E}_2^4 } \right)$ is called timelike Bertrand curve if
there exist an another $C^\infty -$special timelike Frenet curve
${\bar C}$ $\left( {{\bar C}:{\bar L} \to \mathbb{E}_2^4 }
\right)$, distinct from $C$, and a regular $C^\infty -$map  $
\varphi :L \to \bar L$, $\left( {\bar s = \varphi \left( s
\right),\;\frac{{d\varphi \left( s \right)}}{{ds}} \ne 0\quad {\rm
for}\;\;{\rm all}\;\;s \in L} \right)$, such that curve has the
same 1-normal line at each pair of corresponding points ${\bf
c}\left( s \right)$ and ${\bf \bar c}\left( {\bar s} \right) =
{\bf \bar c}\left( {\varphi \left( s \right)} \right)$ under
$\varphi$. Here, $s$ and $\bar s$ are arc length parameters of
timelike curves $C$ and $\bar C$, respectively. In this case,
$\bar C$ is called a timelike Bertrand mate of $C$.
\end{definition}

\begin{theorem}\label{T3.3.}
In $\mathbb{E}_2^4$, no $ C^\infty -$special timelike Frenet curve
is a Bertrand curve.
\end {theorem}

\begin{proof}
Let $ \left( {C,\bar C} \right)$ be a mate of Bertrand curve in
$\mathbb{E}_2^4$. Also, the pair of ${\bf c}\left( s \right)$ and
${\bf \bar c}\left( {\bar s} \right)$ be the corresponding points
of $C$ and $\bar C$, respectively. Then for all $s \in L$ the
curve $\bar C$ is given by
\begin{equation}\label{3.1}
{\bf \bar c}\left( {\bar s} \right) = {\bf c}\left( {\varphi
\left( s \right)} \right) = {\bf c}\left( s \right) + \alpha
\left( s \right){\bf n}_1 \left( s \right)
\end{equation}
where $\alpha$ is $C^\infty -$function on $L$. By differentiating
the equation (\ref{3.1}) with respect to $s$, then
\[
\varphi '\left( s \right)\left. {\frac{{d{\bf \bar c}\left( {\bar
s} \right)}}{{d\bar s}}} \right|_{\bar s = \varphi \left( s
\right)}  = {\bf c'}\left( s \right) + \alpha '\left( s
\right){\bf n}_1 \left( s \right) + \alpha \left( s \right){\bf
n'}_1 \left( s \right)
\]
is obtained. Here and hereafter, the subscript prime denotes the
differentiation with respect to $s$. By using the Frenet formulas,
it is seen that
\[
\varphi '\left( s \right){\bf \bar t}\left( {\varphi \left( s
\right)} \right) = \left( {1 + \alpha \left( s \right)k_1 \left( s
\right)} \right){\bf t}\left( s \right) + \alpha '\left( s
\right){\bf n}_1 \left( s \right) + \alpha \left( s \right)k_2
\left( s \right){\bf n}_2 \left( s \right).
\]
Considering ${\bf n}_1 \left( s \right)$ and ${\bf \bar n}_1
\left( {\varphi \left( s \right)} \right)$ are coincident and
$g\left( {{\bf \bar t}\left( {\varphi \left( s \right)}
\right),{\bf n}_1 \left( s \right)} \right) = 0 $ for all $s \in
L$, we get
\[
\alpha '\left( s \right) = 0,
\]
that is, $\alpha$ is a constant function on $L$. Thus, the
Bertrand mate of $C$ can be rewritten as
\begin{equation}\label{3.2}
{\bf \bar c}\left( {\bar s} \right) = {\bf c}\left( {\varphi
\left( s \right)} \right) = {\bf c}\left( s \right) + \alpha {\bf
n}_1 \left( s \right).
\end{equation}
The differentiation of this last equation with respect to $s$ is
\begin{equation}\label{3.3}
\varphi '\left( s \right){\bf \bar t}\left( {\varphi \left( s
\right)} \right) = \left( {1 + \alpha k_1 \left( s \right)}
\right){\bf t}\left( s \right) + \alpha k_2 \left( s \right){\bf
n}_2 \left( s \right)
\end{equation}
for all $s \in L$.\\
By the fact that $C$ and $\bar C$ are timelike curves, the tangent
vector field of Bertrand mate of $C$ can be given by
\begin{equation}\label{3.4}
{\bf \bar t}\left( {\varphi \left( s \right)} \right) = \cosh
\theta \left( s \right){\bf t}\left( s \right) + \sinh \theta
\left( s \right){\bf n}_2 \left( s \right),
\end{equation}
where $\theta$ is a hyperbolic angle between the timelike tangent
vector fields ${\bf \bar t}\left( {\varphi \left( s \right)}
\right) $ and ${\bf t}\left(s \right) $. According to the
equations (\ref{3.3}) and (\ref{3.4}), the hyperbolic functions
are defined by
\begin{equation}\label{3.5}
\cosh \theta \left( s \right) = {{\left( {1 + \alpha k_1 \left( s
\right)} \right)} \mathord{\left/
 {\vphantom {{\left( {1 + \alpha k_1 \left( s \right)} \right)} {\varphi '\left( s \right)}}} \right.
 \kern-\nulldelimiterspace} {\varphi '\left( s \right)}},
\end{equation}
\begin{equation}\label{3.6}
\sinh \theta \left( s \right) = {{\alpha k_2 \left( s \right)}
\mathord{\left/
 {\vphantom {{\alpha k_2 \left( s \right)} {\varphi '\left( s \right)}}} \right.
 \kern-\nulldelimiterspace} {\varphi '\left( s \right)}}.
\end{equation}
By differentiating the equation (\ref{3.4}) and applying Frenet
formulas,
\[
\begin{array}{l}
  - \varphi '\left( s \right)\bar k_1 \left( {\varphi \left( s \right)} \right){\bf \bar n}_1 \left( {\varphi \left( s \right)} \right) = \frac{{d\left( {\cosh \theta \left( s \right)} \right)}}{{ds}}{\bf t}\left( s \right) \\
\quad\quad \quad \quad \quad \quad \quad \quad \quad \quad \quad \quad  + \left( { - k_1 \left( s \right)\cosh \theta \left( s \right) + k_2 \left( s \right)\sinh \theta \left( s \right)} \right){\bf n}_1 \left( s \right) \\
\quad\quad \quad \quad \quad \quad \quad \quad \quad \quad \quad \quad  + \frac{{d\left( {\sinh \theta \left( s \right)} \right)}}{{ds}}{\bf n}_2 \left( s \right) + k_3 \left( s \right)\sinh \theta \left( s \right){\bf n}_3 \left( s \right) \\
 \end{array}
\]
is obtained. Since  ${\bf n}_1\left(s \right) $ is coincident with
${\bf \bar n}_1 \left( {\varphi \left( s \right)} \right) $, from
the above equation, it is seen that
\[
k_3 \left( s \right)\sinh \theta \left( s \right) = 0.
\]
If we notice that $k_3$ is different from zero, then $\sinh \theta
\left( s \right) = 0$. Considering the equation (\ref{3.6}) and
$k_2 \ne 0$, then  $\alpha=0$. In that time the equation
(\ref{3.2}) implies that $\bar C$ is coincident with $C$. This is
a contradiction. So, the proof is complete.\end{proof}\\
Also, in the same way we can generalize this theorem to
semi-Euclidean space $\mathbb{E}_\nu^{n+1}$ $ \left( {n \ge 3}
\right)$ by taking into consideration the Frenet formulas in
$\mathbb{E}_\nu^{n+1}$. The following result can be easily proved.

\begin{result}\label{R 3.1.}
There is no $C^\infty -$special Frenet curve which is a Bertrand
curve in $\mathbb{E}_\nu^{n+1}$, $ \left( {n \ge 3} \right)$.
\end{result}

\section{Timelike $\left( {1,3} \right) -$Bertrand Curve in $\mathbb{E}_2^4$}\label{Sec4}

According to previous section, there is no Bertrand curve in
semi-Euclidean space. Therefore, let us suggest an idea of a
generalization of Bertrand curve in $\mathbb{E}_2^4$.

\begin{definition}\label{D4.1.}
Let $C$ and $\bar C$ be $C^\infty -$special timelike Frenet curves
in $\mathbb{E}_2^4$ and
\[
\begin{array}{*{20}c}
   {\varphi :L \to \bar L}  \\
   {\quad \quad \quad \quad \;s \to \varphi \left( s \right) = \bar s}  \\
\end{array}\quad \quad \left( {\frac{{d\varphi \left( s \right)}}{{ds}} \ne 0\quad {\rm for}\;\;{\rm all}\;\;s \in L} \right)
\]
be a regular $C^\infty -$map at the corresponding points ${\bf
c}\left( s \right)$ and ${\bf \bar c}\left( {\bar s} \right) =
{\bf \bar c}\left( {\varphi \left( s \right)} \right)$ of $C$ and
$\bar C$, respectively. If the Frenet $\left( {1,3} \right)
-$normal plane at the each point of $C$ is coincident with the
Frenet $\left( {1,3} \right) -$normal plane at the corresponding
point ${\bf \bar c}\left( {\bar s} \right)$ of $\bar C$, then $C$
is called a timelike $\left( {1,3} \right) -$Bertrand curve in
$\mathbb{E}_2^4$. Also, $\bar C$ is called a timelike $\left(
{1,3} \right) -$Bertrand mate of $C$, which is given by
\begin{equation}\label{4.1}
{\bf \bar c}\left( {\bar s} \right) = {\bf c}\left( {\varphi
\left( s \right)} \right) = {\bf c}\left( s \right) + \alpha
\left( s \right){\bf n}_1 \left( s \right) + \beta \left( s
\right){\bf n}_3 \left( s \right),
\end{equation}
where $\alpha$ and $\beta$ are $C^\infty -$functions on $L$.
\end{definition}
The following theorem gives us a characterization of timelike
$\left( {1,3} \right) -$Bertrand curve in $\mathbb{E}_2^4$.

\begin{theorem}\label{T4.1.}
Let $C$ be a $C^\infty -$special timelike Frenet curve with
non-zero curvatures $k_1$, $k_2$, and $k_3$ in $\mathbb{E}_2^4$.
Then $C$ is a timelike $\left( {1,3} \right) -$Bertrand curve if
and only if there exist the constant real numbers $ \alpha
,\;\beta ,\;\gamma ,\;\delta$ satisfying
\[
\begin{array}{l}
 i.)\quad \;\;\alpha k_2 \left( s \right) - \beta k_3 \left( s \right) \ne 0 \\
 ii.)\quad \;\gamma \left( {\alpha k_2 \left( s \right) - \beta k_3 \left( s \right)} \right) - \alpha k_1 \left( s \right) = 1 \\
 iii.)\quad \delta k_3 \left( s \right) =  - \gamma k_1 \left( s \right) + k_2 \left( s \right) \\
 iv.)\quad \left( {\gamma ^2  + 1} \right)k_1 \left( s \right)k_2 \left( s \right) - \gamma \left( {k_1^2 \left( s \right) + k_2^2 \left( s \right) - k_3^2 \left( s \right)} \right) \ne 0 \\
 \end{array}
\]
for all $s \in L$.
\end {theorem}

\begin{proof}
Firstly, let us prove the necessary condition of the theorem. Let
$C$ be a timelike $\left( {1,3} \right) -$Bertrand curve
parametrized by arc length $s$ in $\mathbb{E}_2^4$. Then the
timelike Bertrand mate $\bar C$ of $C$ is given by (\ref{4.1}).
Substituting the Frenet formulas into the differentiation of the
equation (\ref{4.1}) with respect to $s$, we get
\[
\begin{array}{l}
 \varphi '\left( s \right){\bf \bar t}\left( {\varphi \left( s \right)} \right) = \left( {1 + \alpha \left( s \right)k_1 \left( s \right)} \right){\bf t}\left( s \right) + \alpha '\left( s \right){\bf n}_1 \left( s \right) \\
\quad \quad \quad \quad \quad \quad \quad  + \left( {\alpha \left( s \right)k_2 \left( s \right) - \beta \left( s \right)k_3 \left( s \right)} \right){\bf n}_2 \left( s \right) + \beta '\left( s \right){\bf n}_3 \left( s \right) \\
 \end{array}
\]
for all $s \in L$.\\
Since ${\bf n}_1$ and ${\bf \bar n}_1$ principal normal vector
fields of $C$ and $\bar C$ are timelike and the plane spanned by
${\bf n}_1$ and ${\bf n}_3$ is coincident with the plane spanned
by ${\bf \bar n}_1$ and ${\bf \bar n}_3$, then we have
\begin{equation}\label{4.2}
{\bf \bar n}_1 \left( {\varphi \left( s \right)} \right) = \cosh
\theta \left( s \right){\bf n}_1 \left( s \right) + \sinh \theta
\left( s \right){\bf n}_3 \left( s \right)
\end{equation}
\begin{equation}\label{4.3}
{\bf \bar n}_3 \left( {\varphi \left( s \right)} \right) = \sinh
\theta \left( s \right){\bf n}_1 \left( s \right) + \cosh \theta
\left( s \right){\bf n}_3 \left( s \right)
\end{equation}
where $ \sinh \theta \left( s \right) \ne 0$ for all $s \in L$ and
$\theta$ is a hyperbolic angle between the timelike vector fields
${\bf n}_1$ and ${\bf \bar n}_1$.\\
By considering differentiation of the equation (\ref{4.1}) and the
last two equations, we see
\[
\begin{array}{l}
 g\left( {\varphi '\left( s \right){\bf \bar t}\left( {\varphi \left( s \right)} \right),{\bf \bar n}_1 \left( {\varphi \left( s \right)} \right)} \right) = \alpha '\left( s \right)\cosh \theta \left( s \right) + \beta '\left( s \right)\sinh \theta \left( s \right) = 0 \\
 g\left( {\varphi '\left( s \right){\bf \bar t}\left( {\varphi \left( s \right)} \right),{\bf \bar n}_3 \left( {\varphi \left( s \right)} \right)} \right) = \alpha '\left( s \right)\sinh \theta \left( s \right) + \beta '\left( s \right)\cosh \theta \left( s \right) = 0 \\
 \end{array}
\]
and since $ \sinh \theta \left( s \right) \ne 0$,
\[
\alpha '\left( s \right) = 0\quad {\rm and}\quad \beta '\left( s
\right) = 0
\]
that is, $\alpha$ and $\beta$ are constant functions on $L$. Thus,
the equation (\ref{4.1}) is rewritten as
\begin{equation}\label{4.4}
{\bf \bar c}\left( {\bar s} \right) = {\bf c}\left( {\varphi
\left( s \right)} \right) = {\bf c}\left( s \right) + \alpha {\bf
n}_1 \left( s \right) + \beta {\bf n}_3 \left( s \right)
\end{equation}
and its differentiation with respect to $s$ is
\begin{equation}\label{4.5}
\varphi '\left( s \right){\bf \bar t}\left( {\varphi \left( s
\right)} \right) = \left( {1 + \alpha k_1 \left( s \right)}
\right){\bf t}\left( s \right) + \left( {\alpha k_2 \left( s
\right) - \beta k_3 \left( s \right)} \right){\bf n}_2 \left( s
\right)
\end{equation}
Since  ${\bf \bar t}$ and ${\bf t}$ are timelike tangent vector
fields of $C$ and $\bar C$, then
\begin{equation}\label{4.6}
 - \left( {\varphi '\left( s \right)} \right)^2  =  - \left( {1 + \alpha k_1 \left( s \right)} \right)^2  + \left( {\alpha k_2 \left( s \right) - \beta k_3 \left( s \right)}
 \right)^2.
\end{equation}
Thus, we can write
\begin{equation}\label{4.7}
{\bf \bar t}\left( {\varphi \left( s \right)} \right) = \cosh \tau
\left( s \right){\bf t}\left( s \right) + \sinh \tau \left( s
\right){\bf n}_2 \left( s \right)
\end{equation}
and
\[
\cosh \tau \left( s \right) = {{\left( {1 + \alpha k_1 \left( s
\right)} \right)} \mathord{\left/
 {\vphantom {{\left( {1 + \alpha k_1 \left( s \right)} \right)} {\varphi '\left( s \right)}}} \right.
 \kern-\nulldelimiterspace} {\varphi '\left( s \right)}},
\]

\[
\sinh \tau \left( s \right) = {{\left( {\alpha k_2 \left( s
\right) - \beta k_3 \left( s \right)} \right)} \mathord{\left/
 {\vphantom {{\left( {\alpha k_2 \left( s \right) - \beta k_3 \left( s \right)} \right)} {\varphi '\left( s \right)}}} \right.
 \kern-\nulldelimiterspace} {\varphi '\left( s \right)}},
\]
where $\tau$ is a hyperbolic angle between the timelike tangent
vector fields ${\bf \bar t}\left( {\varphi \left( s \right)}
\right)$ and ${\bf t}\left( s \right)$ of $C$ and $\bar C$. By
differentiating the equation (\ref{4.7}) with respect to   and
applying Frenet formulas,
\[
\begin{array}{l}
  - \varphi '\left( s \right)\bar k_1 \left( {\varphi \left( s \right)} \right){\bf \bar n}_1 \left( {\varphi \left( s \right)} \right) = \frac{{d\left( {\cosh \tau \left( s \right)} \right)}}{{ds}}{\bf t}\left( s \right) \\
 \quad \quad \quad \quad \quad \quad \quad \quad \quad \quad \quad \quad  + \left( { - k_1 \left( s \right)\cosh \tau \left( s \right) + k_2 \left( s \right)\sinh \tau \left( s \right)} \right){\bf n}_1 \left( s \right) \\
 \quad \quad \quad \quad \quad \quad \quad \quad \quad \quad \quad \quad + \frac{{d\left( {\sinh \tau \left( s \right)} \right)}}{{ds}}{\bf n}_2 \left( s \right) \\
 \quad \quad \quad \quad \quad \quad \quad \quad \quad \quad \quad \quad  + k_3 \left( s \right)\sinh \tau \left( s \right){\bf n}_3 \left( s \right) \\
 \end{array}
\]
is obtained. Since ${\bf \bar n}_1 \left( {\varphi \left( s
\right)} \right)$ is a linear combination ${\bf n}_1 \left( s
\right)$ and ${\bf n}_3 \left( s \right)$, it easily seen that
\[
\frac{{d\left( {\cosh \tau \left( s \right)} \right)}}{{ds}} =
0\quad {\rm and}\quad \frac{{d\left( {\sinh \tau \left( s \right)}
\right)}}{{ds}} = 0,
\]
that is, $\tau$ is a constant function on $L$ with value $\tau
_0$. Thus, we rewrite the equation (\ref{4.7}) as

\begin{equation}\label{4.8}
{\bf \bar t}\left( {\varphi \left( s \right)} \right) = \cosh \tau
_0 {\bf t}\left( s \right) + \sinh \tau _0 {\bf n}_2 \left( s
\right)
\end{equation}
and
\begin{equation}\label{4.9}
\varphi '\left( s \right)\cosh \tau _0  = 1 + \alpha k_1 \left( s
\right),
\end{equation}

\begin{equation}\label{4.10}
\varphi '\left( s \right)\sinh \tau _0  = \alpha k_2 \left( s
\right) - \beta k_3 \left( s \right),
\end{equation}
for all $s \in L$. According to these last two equations, it is
seen that
\begin{equation}\label{4.11}
\left( {1 + \alpha k_1 \left( s \right)} \right)\sinh \tau _0  =
\left( {\alpha k_2 \left( s \right) - \beta k_3 \left( s \right)}
\right)\cosh \tau _0.
\end{equation}
If $\sinh \tau _0  = 0$, then it satisfies $\cosh \tau _0  = 1$
and ${\bf \bar t}\left( {\varphi \left( s \right)} \right) = {\bf
t}\left( s \right)$. The differentiation of this equality with
respect to $s$ is
\[
 - \bar k_1 \left( {\varphi \left( s \right)} \right)\varphi '\left( s \right){\bf \bar n}_1 \left( {\varphi \left( s \right)} \right) =  - k_1 \left( s \right){\bf n}_1 \left( s
 \right),
\]
that is, ${\bf n}_1 \left( s \right)$ is linear dependence with
${\bf \bar n}_1 \left( {\varphi \left( s \right)} \right)$.
According to Theorem \ref{T3.3.}, this is a contradiction. Thus,
only the case of $\sinh \tau _0  \ne 0$ must be considered. The
equation (\ref{4.10}) satisfies
\[
\alpha k_2 \left( s \right) - \beta k_3 \left( s \right) \ne 0,
\]
that is, the relation given in \textit{i.)} is proved. Since
$\sinh \tau _0 \ne 0$, the equation (\ref{4.11}) can be rewritten
as
\[
\frac{{\cosh \tau _0 }}{{\sinh \tau _0 }}\left( {\alpha k_2 \left(
s \right) - \beta k_3 \left( s \right)} \right) - \alpha k_1
\left( s \right) = 1.
\]
Let us denote the constant value $\left( {\cosh \tau _0 }
\right)\left( {\sinh \tau _0 } \right)^{ - 1}$ by the constant
real number $\gamma$, then $\gamma$ is an element of interval
$\left( { - \infty , - 1} \right) \cup \left( {1,\infty } \right)$
and
\[
\gamma \left( {\alpha k_2 \left( s \right) - \beta k_3 \left( s
\right)} \right) - \alpha k_1 \left( s \right) = 1.
\]
This proves the relation \textit{ii.)} of the theorem.\\
By differentiating the equation (\ref{4.8}) with respect to $s$
and applying Frenet formulas, we have
\[
\begin{array}{l}
 \varphi '\left( s \right)\bar k_1 \left( {\varphi \left( s \right)} \right){\bf \bar n}_1 \left( {\varphi \left( s \right)} \right) = \left( {k_1 \left( s \right)\cosh \tau \left( s \right) - k_2 \left( s \right)\sinh \tau \left( s \right)} \right){\bf n}_1 \left( s \right) \\
\,\,\,\,\,\,\,\,\,\,\,\,\,\,\,\,\,\,\,\,\,\,\,\,\,\,\,\,\,\,\,\,\,\,\,\,\,\,\,\,\,\,\,\,\,\,\,\,\,\,\,\,\,\,\,\,\,\,\,\,\,\,\,\,\,\,\, - k_3 \left( s \right)\sinh \tau \left( s \right){\bf n}_3 \left( s \right) \\
 \end{array}
\]
for all $s \in L$.\\ Taking into consideration the equations
(\ref{4.9}), (\ref{4.10}) and the relation \textit{ii.)}, the
above equality satisfy
\[
\left( {\varphi '\left( s \right)\bar k_1 \left( {\varphi \left( s
\right)} \right)} \right)^2  = \left( {\alpha k_2 \left( s \right)
- \beta k_3 \left( s \right)} \right)^2 \left[ {\left( {\gamma k_1
\left( s \right) - k_2 \left( s \right)} \right)^2  - k_3^2 \left(
s \right)} \right]\left( {\varphi '\left( s \right)} \right)^{ -
2}.
\]
From the equation (\ref{4.6}) and the relation \textit{ii.)}, we
get
\begin{equation}\label{4.12}
\left( {\varphi '\left( s \right)} \right)^2  = \left( {\gamma ^2
- 1} \right)\left( {\alpha k_2 \left( s \right) - \beta k_3 \left(
s \right)} \right)^2.
\end{equation}
Thus, we obtain
\begin{equation}\label{4.13}
\left( {\varphi '\left( s \right)\bar k_1 \left( {\varphi \left( s
\right)} \right)} \right)^2  = \frac{1}{{\gamma ^2  - 1}}\left(
{\left( {\gamma k_1 \left( s \right) - k_2 \left( s \right)}
\right)^2  - k_3^2 \left( s \right)} \right)
\end{equation}
where $\gamma  \in \left( { - \infty , - 1} \right) \cup \left(
{1,\infty } \right)$.\\
From the equations (\ref{4.9}), (\ref{4.10}) and the relation
\textit{ii.)}, we can give
\begin{equation}\label{4.14}
{\bf \bar n}_1 \left( {\varphi \left( s \right)} \right) = \cosh
\eta \left( s \right){\bf n}_1 \left( s \right) + \sinh \eta
\left( s \right){\bf n}_3 \left( s \right)
\end{equation}
where
\begin{equation}\label{4.15}
\cosh \eta \left( s \right) = \frac{{\left( {\alpha k_2 \left( s
\right) - \beta k_3 \left( s \right)} \right)\left( {\gamma k_1
\left( s \right) - k_2 \left( s \right)} \right)}}{{\bar k_1
\left( {\varphi \left( s \right)} \right)\left( {\varphi '\left( s
\right)} \right)^2 }},
\end{equation}

\begin{equation}\label{4.16}
\sinh \eta \left( s \right) =  - \frac{{\left( {\alpha k_2 \left(
s \right) - \beta k_3 \left( s \right)} \right)k_3 \left( s
\right)}}{{\bar k_1 \left( {\varphi \left( s \right)}
\right)\left( {\varphi '\left( s \right)} \right)^2 }},
\end{equation}
for all $s \in L$ and $\eta  \in C^\infty   -$function on $L$. By
differentiating the equation (\ref{4.14}) with respect to $s$ and
applying Frenet formulas, we have
\[
\begin{array}{l}
 \varphi '\left( s \right)\bar k_1 \left( {\varphi \left( s \right)} \right){\bf \bar t}\left( {\varphi \left( s \right)} \right)\\
 \quad \quad \quad \quad \quad + \varphi '\left( s \right)\bar k_2 \left( {\varphi \left( s \right)} \right){\bf \bar n}_2 \left( {\varphi \left( s \right)} \right) = k_1 \left( s \right)\cosh \eta \left( s \right){\bf t}\left( s \right)\\
 \quad \quad \quad \quad \quad \quad \quad \quad \quad \quad \quad \quad \quad \quad \quad \quad \quad + \frac{{d\left( {\cosh \eta \left( s \right)} \right)}}{{ds}}{\bf
n}_1 \left( s \right)\\
\quad \quad \quad \quad \quad \quad \quad \quad \quad \quad \quad
\quad \quad \quad \quad \quad \quad + \left( {k_2 \left( s
\right)\cosh \eta \left( s \right) - k_3 \left( s \right)\sinh
\eta \left( s \right)} \right){\bf n}_2
\left( s \right)\\
\quad \quad \quad \quad \quad \quad \quad \quad \quad \quad \quad
\quad \quad \quad \quad \quad \quad +  \frac{{d\left( {\sinh \tau \left( s \right)} \right)}}{{ds}}{\bf n}_3 \left( s \right) \\
 \end{array}
\]
for all $s \in L$ and this satisfies
\[
\frac{{d\left( {\cosh \eta \left( s \right)} \right)}}{{ds}} =
0\quad {\rm and}\quad \frac{{d\left( {\sinh \eta \left( s \right)}
\right)}}{{ds}} = 0,
\]
that is, $\eta$ is a constant function on $L$ with value $\eta
_0$. Let us denote $\left( {\cosh \tau _0 } \right)\left( {\sinh
\tau _0 } \right)^{ - 1}$ by the constant real number $\delta$,
then $\delta  \in \left( { - \infty , - 1} \right) \cup \left(
{1,\infty } \right)$. The ratio of (\ref{4.15})  and (\ref{4.16})
holds
\[
\delta  =  - \frac{{\gamma k_1 \left( s \right) - k_2 \left( s
\right)}}{{k_3 \left( s \right)}},
\]
that is,
\[
\delta k_3 \left( s \right) =  - \gamma k_1 \left( s \right) + k_2
\left( s \right)
\]
for all $s \in L$ and $\delta  \in \left( { - \infty , - 1}
\right) \cup \left( {1,\infty } \right)$. Thus the relation
\textit{iii.)} of the theorem is obtained. Moreover, we can give

\[
\displaylines{
  \varphi '\left( s \right)\bar k_2 \left( {\varphi \left( s \right)} \right){\bf \bar n}_2 \left( {\varphi \left( s \right)} \right) =  - \varphi '\left( s \right)\bar k_1 \left( {\varphi \left( s \right)} \right){\bf \bar t}\left( {\varphi \left( s \right)} \right) + k_1 \left( s \right)\cosh \eta \left( s \right){\bf t}\left( s \right) \cr
   \,\,\,\,\,\,\,\,\,\,\,\,\,\,\,\,\,\,\,\,\,\,\,\,\,\,\,\,\,\,\,\,\,\,\,\,\,\,\,\,\,\,\,\,\,\,\,\,\,\,\,\,\,\,\,\,\,\,\,\,\,\,\, + \left( {k_2 \left( s \right)\cosh \eta \left( s \right) - k_3 \left( s \right)\sinh \eta \left( s \right)} \right){\bf n}_2 \left( s \right). \cr}
\]
If we substitute the equations (\ref{4.5}), (\ref{4.15}) and
(\ref{4.16}) into the above equality, we obtain

\[
\varphi '\left( s \right)\bar k_2 \left( {\varphi \left( s
\right)} \right){\bf \bar n}_2 \left( {\varphi \left( s \right)}
\right) = \left( {\varphi '\left( s \right)} \right)^{ - 2} \left(
{\bar k_1 \left( {\varphi \left( s \right)} \right)} \right)^{ -
1} \left[ {A\left( s \right){\bf t}\left( s \right) + B\left( s
\right){\bf n}_2 \left( s \right)} \right]
\]
where
\[
\begin{array}{l}
 A\left( s \right) =  - \left( {\varphi '\left( s \right)\bar k_1 \left( {\varphi \left( s \right)} \right)} \right)^2 \left( {1 + \alpha k_1 \left( s \right)} \right) \\
 \quad \quad \quad \,\,\,  + k_1 \left( s \right)\;\left( {\alpha k_2 \left( s \right) - \beta k_3 \left( s \right)} \right)\left( {\gamma k_1 \left( s \right) - k_2 \left( s \right)} \right) \\
 \end{array}
\]
and
\[
\begin{array}{l}
 B\left( s \right) = \left( {\varphi '\left( s \right)\bar k_1 \left( {\varphi \left( s \right)} \right)} \right)^2 \;\left( {\alpha k_2 \left( s \right) - \beta k_3 \left( s \right)} \right) \\
 \quad \quad \quad  + \left( {\alpha k_2 \left( s \right) - \beta k_3 \left( s \right)} \right)\left( {\gamma k_1 \left( s \right) - k_2 \left( s \right)} \right)k_2 \left( s \right) \\
 \quad \quad \quad  + \left( {\alpha k_2 \left( s \right) - \beta k_3 \left( s \right)} \right)k_3^2 \left( s \right) \\
 \end{array}
\]
for all $s \in L$. By the relation \textit{ii.)} and the equation
(\ref{4.13}), $A\left( s \right)$ and $B\left( s \right)$ can be
rewritten as;
\[
A\left( s \right) = \left( {\gamma ^2  - 1} \right)^{ - 1} \left(
{\alpha k_2 \left( s \right) - \beta k_3 \left( s \right)}
\right)\left[ {\left( {\gamma ^2  + 1} \right)k_1 \left( s
\right)k_2 \left( s \right) - \gamma \left( {k_1^2 \left( s
\right) + k_2^2 \left( s \right) - k_3^2 \left( s \right)}
\right)} \right]
\]
and
\[
B\left( s \right) = \gamma \left( {\gamma ^2  - 1} \right)^{ - 1}
\left( {\alpha k_2 \left( s \right) - \beta k_3 \left( s \right)}
\right)\left[ {\left( {\gamma ^2  + 1} \right)k_1 \left( s
\right)k_2 \left( s \right) - \gamma \left( {k_1^2 \left( s
\right) + k_2^2 \left( s \right) - k_3^2 \left( s \right)}
\right)} \right]
\]
By the fact that $\varphi '(s)\overline k _2 (\varphi (s)){\bf
\bar n}_2 (\varphi (s)) \ne 0$ for all $s \in L$, it is proved
that
\[
\left( {\gamma ^2  + 1} \right)k_1 \left( s \right)k_2 \left( s
\right) - \gamma \left( {k_1^2 \left( s \right) + k_2^2 \left( s
\right) - k_3^2 \left( s \right)} \right) \ne 0.
\]
This is the relation \textit{iv.)} of theorem.\\
Now, we will prove the sufficient condition of the theorem.\\
Thus, we assume that $C$ is a timelike $C^\infty-$special Frenet
curve in $\mathbb{E}_2^4$ with curvatures $k_1 ,k_2$ and $k_3$
satisfying the relations \textit{i.), ii.), iii.)} and
\textit{iv.)} for the constant real numbers $\alpha ,\,\,\beta
,\,\,\gamma$, and $\delta$.\\
We define a timelike curve $\bar C$ by
\begin{equation}\label{4.17}
{\bf \bar c}(s) = {\bf c}(s) + \alpha {\bf n}_1 \left( s \right) +
\beta {\bf n}_3 \left( s \right)
\end{equation}
where $s$ is the arc length parameter of $C$.\\
By differentiating the equation (\ref{4.17}) with respect to $s$
and applying Frenet formulas,
\[
\frac{{d{\bf \bar c}\left( s \right)}}{{ds}} = \left( {1 + \alpha
k_1 \left( s \right)} \right){\bf t}\left( s \right) + \left(
{\alpha k_2 \left( s \right) - \beta k_3 \left( s \right)}
\right){\bf n}_2 \left( s \right)
\]
is obtained. Considering the relation \textit{ii.)}, the last
equation is rewritten as;
\[
\frac{{d{\bf \bar c}\left( s \right)}}{{ds}} = \left( {\alpha k_2
\left( s \right) - \beta k_3 \left( s \right)} \right)\left(
{\gamma {\bf t}\left( s \right) + {\bf n}_2 \left( s \right)}
\right)
\]
for all $s \in L$. From the relation \textit{i.)} it is seen that
$\bar C$ is regular curve. Thus, arc length parameter of $\bar C$
denoted by $\overline s$ can be given by
\[
\bar s = \varphi \left( s \right) = \int\limits_0^s {\left\|
{\frac{{d{\bf \bar c}\left( t \right)}}{{dt}}} \right\|dt}
\]
where $\varphi :L \to \overline L$ is a regular map.\\
The differentiation of $\varphi$ with respect to $s$ is
\begin{equation}\label{4.18}
\varphi '\left( s \right) = \varepsilon \sqrt {\gamma ^2  - 1}
\left( {\alpha k_2 \left( s \right) - \beta k_3 \left( s \right)}
\right) > 0
\end{equation}
where
\[
\varepsilon  = \left\{ \begin{array}{l}
 \,\,\,1\,\,\,\,\,\,,\,\,\alpha k_2 \left( s \right) - \beta k_3 \left( s \right) > 0 \\
  - 1\,\,\,\,\,,\,\,\alpha k_2 \left( s \right) - \beta k_3 \left( s \right) < 0. \\
 \end{array} \right.
\]
Also, here we notice that $\gamma  \in \left( { - \infty , - 1}
\right) \cup \left( {1,\infty } \right)$ and $\gamma ^2  - 1 > 0$.
Thus the timelike curve $\bar C$ is rewritten as;
\[
{\bf \bar c}\left( {\bar s} \right) = {\bf \bar c}\left( {\varphi
\left( s \right)} \right) = {\bf c}\left( s \right) + \alpha {\bf
n}_1 \left( s \right) + \beta {\bf n}_3 \left( s \right)
\]
for all $s \in L$. Differentiating this equation with respect to
$s$, we get
\[
\left. {\varphi '\left( s \right)\frac{{d{\bf \bar c}\left( {\bar
s} \right)}}{{d\bar s}}} \right|_{\bar s = \varphi (s)}  = \left(
{\alpha k_2 \left( s \right) - \beta k_3 \left( s \right)}
\right)\left( {\gamma {\bf t}\left( s \right) + {\bf n}_2 \left( s
\right)} \right).
\]
Now, let us define a unit vector field $ {\bf \bar t}$ along $\bar
C$ by $\frac{{d{\bf \bar c}\left( {\bar s} \right)}}{{d\bar s}} $,
then
\begin{equation}\label{4.19}
{\bf \bar t}\left( {\varphi \left( s \right)} \right) =
\varepsilon \left( {\gamma ^2  - 1} \right)^{ -
{\raise0.7ex\hbox{$1$} \!\mathord{\left/
 {\vphantom {1 2}}\right.\kern-\nulldelimiterspace}
\!\lower0.7ex\hbox{$2$}}} \left( {\gamma {\bf t}\left( s \right) +
{\bf n}_2 \left( s \right)} \right).
\end{equation}
By differentiating the equation (\ref{4.19}) with respect to $s$
and using Frenet formulas,
\[
\left. {\varphi '\left( s \right)\frac{{d{\bf \bar t}\left(
{\varphi \left( s \right)} \right)}}{{d\bar s}}} \right|_{\bar s =
\varphi (s)}  = \varepsilon \left( {\gamma ^2  - 1} \right)^{ -
{\raise0.7ex\hbox{$1$} \!\mathord{\left/
 {\vphantom {1 2}}\right.\kern-\nulldelimiterspace}
\!\lower0.7ex\hbox{$2$}}} \left[ { - \left( {\gamma k_1 \left( s
\right) - k_2 \left( s \right)} \right){\bf n}_1 \left( s \right)
+ k_3 \left( s \right){\bf n}_3 \left( s \right)} \right]
\]
and
\[
\left\| {\left. {\varphi '\left( s \right)\frac{{d{\bf \bar
t}\left( {\varphi \left( s \right)} \right)}}{{d\bar s}}}
\right|_{\bar s = \varphi (s)} } \right\| = \frac{{\sqrt {\left| {
- \left( {\gamma k_1 \left( s \right) - k_2 \left( s \right)}
\right)^2  + k_3^2 \left( s \right)} \right|} }}{{\varphi
'(s)\sqrt {\gamma ^2  - 1} }}.
\]
By relations $iii.)$, it is seen that
\[
\sqrt {\left| { - \left( {\gamma k_1 \left( s \right) - k_2 \left(
s \right)} \right)^2  + k_3^2 \left( s \right)} \right|}  = \sqrt
{\left| { - \delta ^2  + 1} \right|k_3^2 \left( s \right)}
\]
and we notice that $\delta  \in \left( { - \infty , - 1} \right)
\cup \left( {1,\infty } \right)$ and $\gamma ^2  - 1 > 0$.\\
Thus, we can write
\[
\left\| {\left. {\frac{{d{\bf \bar t}\left( {\varphi \left( s
\right)} \right)}}{{d\bar s}}} \right|_{\bar s = \varphi (s)} }
\right\| = \frac{{\sqrt {\delta ^2  - 1} \,k_3^2 \left( s
\right)}}{{\varphi '\left( s \right)\sqrt {\gamma ^2  - 1} }}.
\]
Since $k_3 \left( s \right) > 0$ and $\varphi '\left( s \right)
> 0$ for all $s \in L$, we obtain
\begin{equation}\label{4.20}
\overline k _1 \left( {\varphi \left( s \right)} \right) = \left\|
{\left. {\frac{{d{\bf \bar t}\left( {\varphi \left( s \right)}
\right)}}{{d\bar s}}} \right|_{\bar s = \varphi (s)} } \right\| >
0.
\end{equation}
Thus, ${\bf \bar n}_1$ timelike unit vector field can be defined
by
\[
\begin{array}{l}
 {\bf \bar n}\left( {\bar s} \right) = {\bf \bar n}_1 \left( {\varphi \left( s \right)} \right) \\
 \quad \quad \; = \, - \frac{1}{{\bar k_1 \left( {\varphi \left( s \right)} \right)}}{\bf \bar t}\left( {\varphi \left( s \right)} \right) \\
 \quad \quad \; = \,\frac{{\left( {\gamma k_1 \left( s \right) - k_2 \left( s \right)} \right){\bf n}_1 \left( s \right) - k_3 \left( s \right){\bf n}_3 \left( s \right)}}{{\varepsilon \sqrt {\left( {\gamma k_1 \left( s \right) - k_2 \left( s \right)} \right)^2  - k_3^2 \left( s \right)} }}. \\
 \end{array}
\]

Also, ${\bf \bar n}_1$  can be given in the form
\begin{equation}\label{4.21}
{\bf \bar n}_1 \left( {\varphi \left( s \right)} \right) = \cosh
\xi \left( s \right){\bf n}_1 \left( s \right) + \sinh \xi \left(
s \right){\bf n}_3 \left( s \right)
\end{equation}
where
\begin{equation}\label{4.22}
\cosh \xi  = \frac{{\gamma k_1 \left( s \right) - k_2 \left( s
\right)}}{{\varepsilon \sqrt {\left( {\gamma k_1 \left( s \right)
- k_2 \left( s \right)} \right)^2  - k_3^2 \left( s \right)} }},
\end{equation}

\begin{equation}\label{4.23}
\sinh \xi  = \frac{{ - k_3 \left( s \right)}}{{\varepsilon \sqrt
{\left( {\gamma k_1 \left( s \right) - k_2 \left( s \right)}
\right)^2  - k_3^2 \left( s \right)} }},
\end{equation}
for all $s \in L$. Here $\xi$ is a $C^\infty$-function on $L$.\\
By differentiating (\ref{4.21})  with respect to $s$ and using the
Frenet formulas, we get
\[
\begin{array}{l}
 \left\| {\left. {\frac{{d{\bf \bar n}_1 \left( {\varphi \left( s \right)} \right)}}{{d\bar s}}} \right|_{\bar s = \varphi (s)} } \right\| = k_1 \left( s \right)\cosh \xi \left( s \right){\bf t}\left( s \right) + \frac{{d\left( {\cosh \xi \left( s \right)} \right)}}{{ds}}{\bf n}_1 \left( s \right) \\
\quad  \quad \quad \quad \quad \quad \quad \quad \; + \left( {k_2 \left( s \right)\cosh \xi \left( s \right) - k_3 \left( s \right)\sinh \xi \left( s \right)} \right){\bf n}_2 \left( s \right) \\
\quad  \quad \quad \quad \quad \quad \quad \quad \; + \frac{{d\left( {\sinh \xi \left( s \right)} \right)}}{{ds}}{\bf n}_3 (s). \\
 \end{array}.
\]
The differentiation of the relation \textit{iii.)} with respect to
$s$ is
\begin{equation}\label{4.24}
\left( {\gamma k'_1 \left( s \right) - k'_2 \left( s \right)}
\right)\,k_3 \left( s \right) + \left( {\gamma k_1 \left( s
\right) - k_2 \left( s \right)} \right)\,k'_3 \left( s \right) = 0
\end{equation}
Substituting the equation (\ref{4.24}), we get
\[
\frac{{d\left( {\cosh \xi \left( s \right)} \right)}}{{ds}} =
0\quad ,\quad \frac{{d\left( {\sinh \xi \left( s \right)}
\right)}}{{ds}} = 0,
\]
that is, $\xi$ is a constant function on $L$ with value $\xi _0$.
Thus, we write
\begin{equation}\label{4.25}
\cosh \xi _0  = \frac{{\gamma k_1 \left( s \right) - k_2 \left( s
\right)}}{{\varepsilon \sqrt {\left( {\gamma k_1 \left( s \right)
- k_2 \left( s \right)} \right)^2  - k_3^2 \left( s \right)} }},
\end{equation}
\begin{equation}\label{4.26}
\sinh \xi _0  = \frac{{ - k_3 \left( s \right)}}{{\varepsilon
\sqrt {\left( {\gamma k_1 \left( s \right) - k_2 \left( s \right)}
\right)^2  - k_3^2 \left( s \right)} }}.
\end{equation}
Then, from the equation (\ref{4.2}), it satisfies
\begin{equation}\label{4.27}
{\bf \bar n}_1 \left( {\varphi \left( s \right)} \right) = \cosh
\xi _0 \left( s \right){\bf n}_1 \left( s \right) + \sinh \xi _0
\left( s \right){\bf n}_3 \left( s \right).
\end{equation}
By considering the equations (\ref{4.19}) and (\ref{4.20}), we
obtain
\[
\overline k _1 \left( {\varphi \left( s \right)} \right)\overline
{\bf t} \left( {\varphi \left( s \right)} \right) = \frac{{\left(
{\gamma k_1 \left( s \right) - k_2 \left( s \right)} \right)^2  -
k_3^2 \left( s \right)}}{{\varepsilon \varphi '\left( s
\right)\left( {\gamma ^2  - 1} \right)\sqrt {\left( {\gamma k_1
\left( s \right) - k_2 \left( s \right)} \right)^2  - k_3^2 \left(
s \right)} }}\left( {\gamma {\bf t}\left( s \right) + {\bf n}_2
\left( s \right)} \right).
\]
Also, by substituting the equations (\ref{4.25}) and (\ref{4.26})
into equation (\ref{4.20}), we get
\[
\begin{array}{l}
 \left. {\frac{{d{\bf \bar n}_1 \left( {\varphi \left( s \right)} \right)}}{{d\bar s}}} \right|_{\bar s = \varphi (s)}  = \frac{{k_1 \left( s \right)\left( {\gamma k_1 \left( s \right) - k_2 \left( s \right)} \right)}}{{\varepsilon \varphi '\left( s \right)\left( {\gamma ^2  - 1} \right)\sqrt {\left( {\gamma k_1 \left( s \right) - k_2 \left( s \right)} \right)^2  - k_3^2 \left( s \right)} }}{\bf t}\left( s \right) \\
 \,\,\,\,\,\,\,\,\,\,\,\,\,\,\,\,\,\,\,\,\,\,\,\,\,\,\,\,\,\,\,\,\,\,\,\,\,\,\,\,\,\,\,\,\,\,\,\,+ \frac{{k_2 \left( s \right)\left( {\gamma k_1 \left( s \right) - k_2 \left( s \right)} \right) + k_3 (s)}}{{\varepsilon \varphi '\left( s \right)\sqrt {\left( {\gamma k_1 \left( s \right) - k_2 \left( s \right)} \right)^2  - k_3^2 \left( s \right)} }}{\bf n}_2 \left( s \right) \\
 \end{array}
\]
for $s \in L$. By the last two equations, we obtain
\[
\left. {\frac{{d{\bf \bar n}_1 \left( {\varphi \left( s \right)}
\right)}}{{d\bar s}}} \right|_{\bar s = \varphi (s)}  - \overline
k _1 \left( {\varphi \left( s \right)} \right)\,\overline {\bf t}
\left( {\varphi \left( s \right)} \right) = \frac{{P\left( s
\right)}}{{R\left( s \right)}}{\bf t}\left( s \right) +
\frac{{Q\left( s \right)}}{{R\left( s \right)}}{\bf n}_2 \left( s
\right)
\]
where
\[
\begin{array}{l}
 P\left( s \right) = \left( {\gamma ^2  + 1} \right)k_1 \left( s \right)k_2 \left( s \right) - \gamma \left( {k_1^2 \left( s \right) + k_2^2 \left( s \right) - k_3^2 \left( s \right)} \right), \\
 Q\left( s \right) = \gamma \left[ {\left( {\gamma ^2  + 1} \right)k_1 \left( s \right)k_2 \left( s \right) - \gamma \left( {k_1^2 \left( s \right) + k_2^2 \left( s \right) - k_3^2 \left( s \right)} \right)} \right], \\
 R\left( s \right) = \varepsilon \varphi '\left( s \right)\left( {\gamma ^2  - 1} \right)\sqrt {\left( {\gamma k_1 \left( s \right) - k_2 \left( s \right)} \right)^2  - k_3^2 \left( s \right)}  \ne 0. \\
 \end{array}
\]
By the fact that
\[
\overline k _2 \left( {\varphi \left( s \right)} \right)\, =
\left\| {\left. {\frac{{d{\bf \bar n}_1 \left( {\varphi \left( s
\right)} \right)}}{{d\bar s}}} \right|_{\bar s = \varphi (s)}  -
\overline k _1 \left( {\varphi \left( s \right)}
\right)\,\overline {\bf t} \left( {\varphi \left( s \right)}
\right)} \right\| > 0
\]
for all $s \in L$, we see
\[
\overline k _2 \left( {\varphi \left( s \right)} \right) =
\frac{{\left| {\left( {\gamma ^2  + 1} \right)k_1 \left( s
\right)k_2 \left( s \right) - \gamma \left( {k_1^2 \left( s
\right) + k_2^2 \left( s \right) - k_3^2 \left( s \right)}
\right)} \right|}}{{\varphi '\left( s \right)\sqrt {\left( {\gamma
^2  - 1} \right)\left[ {\left( {\gamma k_1 \left( s \right) - k_2
\left( s \right)} \right)^2  - k_3^2 \left( s \right)} \right]\,}
}}.
\]
Thus, we can define a unit vector field ${\bf \bar n}_2 \left(
{\bar s} \right)$ along $\overline C$ by
\[
\begin{array}{l}
 {\bf \bar n}_2 \left( {\bar s} \right) = {\bf \bar n}_2 \left( {\varphi \left( s \right)} \right) \\
 \quad \quad \;\;\; = \,\,\frac{1}{{\bar k_2 \left( {\varphi \left( s \right)} \right)}}\left[ {\left. {\frac{{d{\bf \bar n}_1 \left( {\varphi \left( s \right)} \right)}}{{d\bar s}}} \right|_{\bar s = \varphi (s)}  - \overline k _1 \left( {\varphi \left( s \right)} \right)\,\overline {\bf t} \left( {\varphi \left( s \right)} \right)} \right] \\
 \end{array},
\]
such that
\begin{equation}\label{4.28}
{\bf \bar n}_2 \left( {\varphi \left( s \right)} \right) =
\frac{1}{{\varepsilon \sqrt {\gamma ^2  - 1} }}\left( {{\bf
t}\left( s \right) + \gamma {\bf n}_2 \left( s \right)} \right).
\end{equation}
Also, since ${\bf \bar n}_3 \left( {\varphi \left( s \right)}
\right) = \sinh \xi _0 \left( s \right){\bf n}_1 \left( s \right)
+ \cosh \xi _0 \left( s \right){\bf n}_3 \left( s \right)$ for all
$s \in L$, another unit vector field ${\bf \bar n}_3$ along
$\overline C$ can
\begin{equation}\label{4.29}
\begin{array}{l}
 {\bf \bar n}_3 \left( {\bar s} \right) = {\bf \bar n}_3 \left( {\varphi \left( s \right)} \right) \\
 \,\,\,\,\,\,\,\,\,\,\,\,\,\,\,\; = \,\frac{1}{{\varepsilon \sqrt {\left( {\gamma k_1 \left( s \right) - k_2 \left( s \right)} \right)^2  - k_3^2 \left( s \right)} }}\left( { - k_3 \left( s \right){\bf n}_1 \left( s \right) + \left( {\gamma k_1 \left( s \right) - k_2 \left( s \right)} \right){\bf n}_3 \left( s \right)} \right). \\
 \end{array}
\end{equation}
Now, from the equations (\ref{4.19}), (\ref{4.27}), (\ref{4.28})
and (\ref{4.29}), it is seen that
\[
\det \left[ {{\bf \bar t}\left( {\varphi \left( s \right)}
\right),\,{\bf \bar n}_1 \left( {\varphi \left( s \right)}
\right),\,\,{\bf \bar n}_2 \left( {\varphi \left( s \right)}
\right),\,{\bf \bar n}_3 \left( {\varphi \left( s \right)}
\right)} \right] = \det \left[ {{\bf t}\left( s \right),\,{\bf
n}_1 \left( s \right),\,\,{\bf n}_2 \left( s \right),\,{\bf n}_3
\left( s \right)} \right] = 1.
\]
${\bf \bar t},\,\;{\bf \bar n}_1 ,\,\,{\bf \bar n}_2$ and ${\bf
\bar n}_3$ are mutually orthogonal vector fields satisfying
\[
g\left( {{\bf \bar t}\left( {\varphi \left( s \right)}
\right),{\bf \bar t}\left( {\varphi \left( s \right)} \right)}
\right) = g\left( {{\bf \bar n}_1 \left( {\varphi \left( s
\right)} \right),{\bf \bar n}_1 \left( {\varphi \left( s \right)}
\right)} \right) =  - 1,
\]
\[
\,g\left( {{\bf \bar n}_2 \left( {\varphi \left( s \right)}
\right),{\bf \bar n}_2 \left( {\varphi \left( s \right)} \right)}
\right) = g\left( {{\bf \bar n}_3 \left( {\varphi \left( s
\right)} \right),{\bf \bar n}_3 \left( {\varphi \left( s \right)}
\right)} \right) = 1.
\]
Thus the tetrahedron $\left\{ {{\bf \bar t},\,{\bf \bar n}_1
,\,\,{\bf \bar n}_2 ,\,{\bf \bar n}_3 } \right\}$ along $\overline
C$ is an orthonormal frame where ${\bf \bar t}$ and $\,{\bf \bar
n}_1$ are timelike vector fields, $\,{\bf \bar n}_2$ and $\,{\bf
\bar n}_3$ are spacelike vector fields.\\
On the other hand, by considering the equations (\ref{4.29}) and
the differentiation of the equation (\ref{4.28}), we obtain
\[
\begin{array}{l}
 \overline k _3 \left( {\varphi \left( s \right)} \right) =  - \left\langle {\left. {\frac{{d{\bf \bar n}_2 \left( {\varphi \left( s \right)} \right)}}{{d\bar s}}} \right|_{\bar s = \varphi (s)} ,{\bf \bar n}_3 \left( {\varphi \left( s \right)} \right)} \right\rangle  \\
 \,\,\,\,\,\,\,\,\,\,\,\,\,\,\,\,\,\,\,\, = \,\frac{{\sqrt {\gamma ^2  - 1} \,k_1 \left( s \right)k_3 \left( s \right)}}{{\varphi '\left( s \right)\sqrt {\left( {\gamma k_1 \left( s \right) - k_2 \left( s \right)} \right)^2  - k_3^2 \left( s \right)\,} }} > 0 \\
 \end{array}
\]
for all $s \in L$. Therefore, $\overline C$ is a
$C^\infty-$special curve in $\mathbb{E}_2^4$ and the Frenet
(1,3)-normal plane at the corresponding point ${\bf \bar c}\left(
{\bar s} \right) = {\bf \bar c}\left( {\varphi \left( s \right)}
\right)$ of $\overline C$. Thus, $(C,\overline C )$ is a mate of
(1,3)-Bertrand curve in
$\mathbb{E}_2^4$. \\
Finally, the proof of the theorem is completed.
\end{proof}\\

\end {document}